\documentclass[reqno,10pt]{amsart}
\usepackage{amsmath}
\usepackage{amssymb}
\usepackage{amsfonts}

\usepackage{graphicx}
\usepackage[abs]{overpic}
\usepackage{caption}
\usepackage{subcaption}
\usepackage{wrapfig}
\usepackage{float}
\usepackage{graphicx}
\usepackage{physics}
\usepackage{esint} 
\usepackage{braket} 

\usepackage{tikz}
\usetikzlibrary{decorations.pathreplacing,calc}
\def\centerarc[#1](#2)(#3:#4:#5)
{ \draw[#1] ($(#2)+({#5*cos(#3)},{#5*sin(#3)})$) arc (#3:#4:#5); }

\definecolor{blue_links}{RGB}{13,0,180} 

\usepackage{hyperref}
\hypersetup{
    colorlinks=true, 
    linktoc=all,     
    linkcolor=blue_links,  
    citecolor=blue_links,
    urlcolor=blue_links,
}

\usepackage{mathtools}

\usepackage{xcolor}

\newtheorem{theorem}{Theorem}[section]
\newtheorem{lemma}[theorem]{Lemma}

\newtheorem{corollary}[theorem]{Corollary}

\newtheorem{remark}[theorem]{Remark}

\newtheorem*{theorem*}{Theorem}

\newcommand{\N}{\mathbb{N}}

\newcommand{\R}{\mathbb{R}}

\def\eps{\varepsilon}

\def\weakly{\rightharpoonup}

\def\Xint#1{\mathchoice
    {\XXint\displaystyle\textstyle{#1}}%
    {\XXint\textstyle\scriptstyle{#1}}%
    {\XXint\scriptstyle\scriptscriptstyle{#1}}%
    {\XXint\scriptscriptstyle\scriptscriptstyle{#1}}%
\!\int}
\def\XXint#1#2#3{{\setbox0=\hbox{$#1{#2#3}{\int}$}
\vcenter{\hbox{$#2#3$}}\kern-.5\wd0}}

\def\dashint{\Xint-}

\usepackage{color}

\numberwithin{equation}{section}

\begin{document} 

\title[Compactness for $GSBV^p$ via concentration-compactness]{Compactness for $GSBV^p$ via concentration-compactness}
\author[W. Feldman]{William M Feldman} 
\address[William Feldman]{Department of Mathematics, University of Utah, Salt Lake City, USA}
\email{feldman@math.utah.edu}
\author[K. Stinson] {Kerrek Stinson} 
\address[Kerrek Stinson]{Department of Mathematics, University of Utah, Salt Lake City, USA}
\email{kerrek.stinson@utah.edu}

\subjclass[2010]{49J45, 70G75,   74B99, 74G65, 74R10}
\keywords{concentration-compactness, fracture, free discontinuity problem }

\begin{abstract}   
Motivated by variational models for fracture, we provide a new proof of compactness for $GSBV^p$ functions without a priori bounds on the function itself. Our proof is based on the classical idea of concentration-compactness, making it transparent in strategy and simple in implementation. Further, so far as we are aware, this is the first time the connection to concentration-compactness has been made explicit for problems in fracture mechanics.
\end{abstract}

\maketitle

\section{Introduction}

Variational models for fracture are based on the competition of two interconnected energies: a bulk elastic energy, due to stretching of the unbroken material, and a dissipation term, arising from the creation of a crack \cite{francfortMarigo98}. In antiplanar elasticity, the energy might take the form of 
\begin{equation}\label{eqn:MSnoFid}
E[u] := \int_{\Omega'} |\nabla u|^2 \, dx +\mathcal{H}^{N-1}(J_u),
\end{equation}
for a material body $\Omega' \subset \R^N$ with displacement $u\in GSBV^2(\Omega')$. 
 Critically, \eqref{eqn:MSnoFid} does not directly control the displacement, and sequences $u_n$ with uniformly bounded energy \eqref{eqn:MSnoFid} may have broken off pieces {that} `travel to infinity'; i.e., for $\Omega' = (-1,1)\times (0,1)$ the sequence $u_n : = n\chi_{(0,1)\times (0,1)}$ has bounded energy, and the piece $(0,1)\times (0,1)$ travels off to infinity.
 One way to think of this is that the energy associated to the current configuration $u(\Omega\setminus J_u)$ is invariant under {piecewise-constant translations} which preserve the shape of each broken piece. To cast this in a more mathematical framework, we can introduce a symmetry group of piecewise-constant translations: for any disjoint collection of sets $\mathcal{S} = \{S_j\}_{j=1}^\infty$ of $\Omega'$ with $\partial^*S_j\subset J_u$ and constants $a_j\in \R$,  the energy $E$ is invariant under the piecewise-constant translation $\sum_{j} a_j \chi_{S_j}$, in the sense that $E[u] = E[u + \sum_{j} a_j \chi_{S_j}]$.
 
 But this is perhaps the key point of Lions' concentration-compactness \cite{lions_CC_84}: to address a loss of compactness due to symmetry groups. While many of the prototypical applications consider invariance of energies due to domain translations (e.g., $\int_{\R^d} |\nabla u|^2 \, dx = \int_{\R^d} |\nabla u(x+a)|^2 \, dx$ for all $a\in \R^d$), we show that these classical ideas provide a fruitful way of looking at compactness for fracture problems where invariance is due to piecewise-constant translations of the current configuration.

 The purpose of this paper is to prove the following compactness theorem, which also accounts for imposed displacement boundary conditions via the function $h$. {See the end of this section for notation.}

\begin{theorem}\label{thm:main}
Let $1<p<\infty$, integers $d,N\geq 1$, and $\Omega\subset \Omega' \subset \R^N$ be Lipschitz domains.  Let $(u_n)_{n\in \N} \in GSBV^p({\Omega'}; \R^d)$ be such that 
\begin{equation}\label{eqn:energyBound}
\sup_{n\in \N}\left(\int_{{\Omega'}} |\nabla u_n|^p \, dx + \mathcal{H}^{N-1}(J_{u_n})\right)<\infty \quad \text{ and } \quad u_n = h \text{ on } \Omega'\setminus \Omega,
\end{equation}
where $h\in W^{1,p}(\Omega';\R^d).$
Then there exists a subsequence of $(u_n)_{n\in \N}$ (not relabeled), a collection of disjoint sets of finite perimeter $\mathcal{S}_n:=(S_j^n)_{j=0}^{\infty}$ contained in $\Omega$, vectors $(a_j^n)_{j=1}^\infty \R^d$, and a limit function $u\in GSBV^p(\Omega';\R^d)$ such that the following holds:
\begin{enumerate}
\item $u_n - \sum_{j =1}^\infty a_j^n \chi_{S_j^n} + (h-u_n)\chi_{S_0^n}\to u$ in measure in $\Omega'$,
\item $\nabla u_n \weakly \nabla u$ in $L^p(\Omega';\R^{d\times N})$,
\item $\mathcal{H}^{N-1} (J_u)\leq \liminf_{n\to \infty}\mathcal{H}^{N-1}(J_{u_n})$,
\item $\mathcal{H}^{N-1}(\cup_{j=0}^\infty \partial^* S_{j}^n \setminus J_{u_n})\to 0$ and $\mathcal{L}^{N}(S_0^n)\to 0$ as $n\to \infty$,
\item and $|a_j^n - a_i^n|\to \infty$ for $i\neq j$ as $n\to \infty.$
\end{enumerate}
\end{theorem}

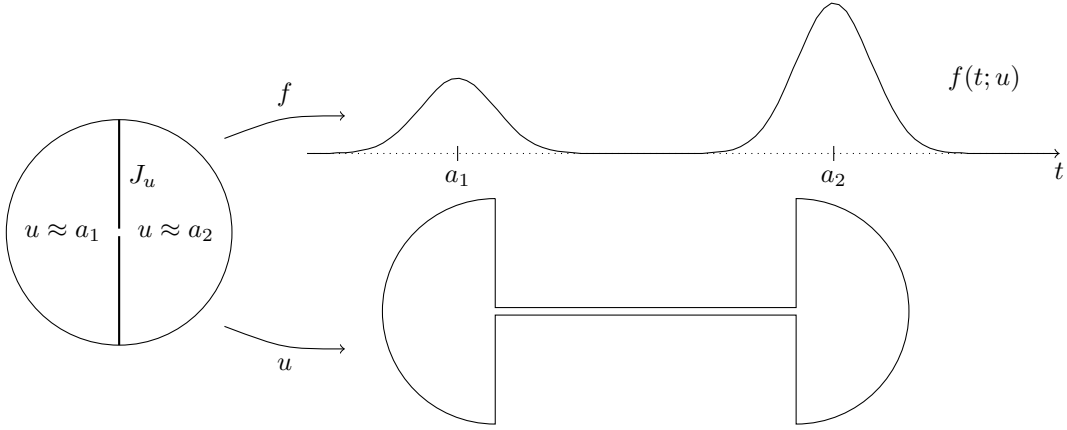
\begin{figure}
\begin{tikzpicture}
    \def\r{1.5}
    \def\gap{.05}

    \begin{scope}[shift={(-\r-5-1,0)}]
    
    \draw (0,0) circle (\r);

    \draw[thick] (0,\r) -- (0,\gap);
    \draw[thick] (0,-\r) -- (0,-\gap);
    \node at (-{.5*\r},0) {$u \approx a_1$};
    \node at ({.5*\r},0) {$u \approx a_2$};
    \node[right] at (0,{\r/2}) {$J_u$};

    \draw[->] (\r-.1,{\r/2+.3+.2}) .. controls (\r+.7,{\r/2+.3+.5})  .. (\r+1.5,{\r/2+.3+.5});
    \node[above] at (\r+.7,{\r/2+.3+.5}) {$f$};
    \draw[->] (\r-.1,{-\r/2-.3-.2}) .. controls (\r+.7,{-\r/2-.3-.5}) .. (\r+1.5,{-\r/2-.3-.5});
    \node[below] at (\r+.7,{-\r/2-.3-.5}) {$u$};
    \end{scope}

    \begin{scope}[shift={(0,{\r/2+.3})}]

    \draw[dotted,->] (-5,0) -- (5,0);
    \node[below] at (5,0) {$t$};

    \draw (-5.00,0.00)--(-4.90,0.00)--(-4.80,0.00)--(-4.70,0.00)--(-4.60,0.01)--(-4.49,0.01)--(-4.39,0.02)--(-4.29,0.04)--(-4.19,0.06)--(-4.09,0.09)--(-3.99,0.14)--(-3.89,0.21)--(-3.79,0.29)--(-3.69,0.39)--(-3.59,0.50)--(-3.48,0.62)--(-3.38,0.74)--(-3.28,0.85)--(-3.18,0.94)--(-3.08,0.99)--(-2.98,1.00)--(-2.88,0.97)--(-2.78,0.91)--(-2.68,0.81)--(-2.58,0.70)--(-2.47,0.58)--(-2.37,0.46)--(-2.27,0.35)--(-2.17,0.25)--(-2.07,0.18)--(-1.97,0.12)--(-1.87,0.08)--(-1.77,0.05)--(-1.67,0.03)--(-1.57,0.02)--(-1.46,0.01)--(-1.36,0.00)--(-1.26,0.00)--(-1.16,0.00)--(-1.06,0.00)--(-0.96,0.00)--(-0.86,0.00)--(-0.76,0.00)--(-0.66,0.00)--(-0.56,0.00)--(-0.45,0.00)--(-0.35,0.00)--(-0.25,0.00)--(-0.15,0.00)--(-0.05,0.00)--(0.05,0.00)--(0.15,0.00)--(0.25,0.00)--(0.35,0.01)--(0.45,0.02)--(0.56,0.03)--(0.66,0.05)--(0.76,0.09)--(0.86,0.15)--(0.96,0.23)--(1.06,0.34)--(1.16,0.49)--(1.26,0.67)--(1.36,0.89)--(1.46,1.13)--(1.57,1.37)--(1.67,1.60)--(1.77,1.80)--(1.87,1.93)--(1.97,2.00)--(2.07,1.98)--(2.17,1.89)--(2.27,1.72)--(2.37,1.51)--(2.47,1.27)--(2.58,1.03)--(2.68,0.80)--(2.78,0.60)--(2.88,0.43)--(2.98,0.29)--(3.08,0.19)--(3.18,0.12)--(3.28,0.07)--(3.38,0.04)--(3.48,0.02)--(3.59,0.01)--(3.69,0.01)--(3.79,0.00)--(3.89,0.00)--(3.99,0.00)--(4.09,0.00)--(4.19,0.00)--(4.29,0.00)--(4.39,0.00)--(4.49,0.00)--(4.60,0.00)--(4.70,0.00)--(4.80,0.00)--(4.90,0.00)--(5.00,0.00);
    \node at (4,1) {$f(t;u)$};

    \draw (-3,.1)--(-3,-.1);
    \node[below] at (-3,-.1) {$a_1$};

    \draw (2,.1)--(2,-.1);
    \node[below] at (2,-.1) {$a_2$};
    \end{scope}
    
    \begin{scope}[shift={(0,{-\r/2-.3})}]
        \draw  (-2.5,\r) arc (90:270:\r) -- (-2.5,-\gap) -- (1.5,-\gap) -- (1.5,-\r) arc (-90:90:\r) -- (1.5,\gap) -- (-2.5,\gap) -- (-2.5,\r);
    \end{scope}
\end{tikzpicture}
\caption{Left: The material domain with crack {$J_u$}. Above right: {The function} $f(t;u)$ from \eqref{eqn:fn-intro} measures concentrations in the range. Below right: The current configuration as deformed/displaced by $u(x)$.}
\label{fig:cc-function}
\end{figure}

Let us be a bit more specific {in explaining} the application of the concentration-compactness principle to this case. 
 For simplicity, we discuss the scalar case as in \eqref{eqn:MSnoFid}; in fact, we will later reduce the proof of Theorem \ref{thm:main} to the scalar setting. 
 For $u\in {GSBV^2(\Omega')}$, we consider {\emph{the concentration function}}
\begin{equation}\label{eqn:fn-intro}
f(t;u) : = \mathcal{H}^{N-1}(\partial^*\{u>t\}\setminus J_{u}) + \sum_{\pm} \mathcal{H}^{N-1}(\{t-1<u^{\pm}<t+1\}\cap (J_{u}\cup \partial \Omega')),
\end{equation}
where $u^{\pm}$ denotes the two possible trace values of $u$ on the rectifiable set $J_{u}\cup \partial \Omega'.$ The $L^1(\R)$ integral of $f(t;u)$ is controlled by the energy $E[u]$ via {the coarea} formula.  The idea is that ${f(t;u)}$ naturally captures all the values in the range where $u$ concentrates. The first term captures {the} displacement {by looking at} stretching in unbroken components and the second captures {the} displacement via breaking. See Figure~\ref{fig:cc-function}.  Lions' concentration{-}compactness compactness principle is naturally suited to sequences of functions with bounded ``mass", decomposing such sequences into countably many bubbles of ``mass" concentration and a ``vanishing" component, see Section~\ref{sec:CC} below. The bubbles picked by concentration{-}compactness yield the rigid translations $a_j$ in Theorem~\ref{thm:main}.  The ``vanishing" case is also quite interesting in this context,  see Remark~\ref{rmk:vanishExample} below. We show that the concentration{-}compactness ``vanishing" scenario can occur via many small fractures, and can carry a nontrivial energy, but must occur on a set of vanishing $N$-dimensional measure in the domain, see Lemma~\ref{lem:vanishing} below.

This kind of compactness result was first obtained by Friedrich in \cite{friedrich_GSBVp} using a piecewise Poincar\'e inequality along with a {technically sophisticated} argument to group elements of Caccioppoli partitions together.
In comparison, we find our strategy to be simple because we immediately generate the correct Caccioppoli partition via the concentration bubbles. Our proof does also rely on the coarea formula and, as such, cannot be immediately generalized to linearized elasticity. Nevertheless, Theorem \ref{thm:main} is applicable in the nonlinear setting of finite elasticity, and in particular a direct corollary of it gives existence of minimizers to the energy
\begin{equation}\label{eqn:nonlinEnergy}
\int_\Omega W(\nabla u)\, d x + \mathcal{H}^{N-1}(J_u) \quad \hbox{ with } \  u =h \text{ on }\Omega'\setminus \Omega, 
\end{equation}
where $W$ is the quasiconvexification of ${\rm dist}^2(\nabla u,SO(N)).$ Precisely, if one takes a minimizing sequence $u_n$ of the energy \eqref{eqn:nonlinEnergy}, Theorem \ref{thm:main} gives a new minimizing sequence $v_n : = u_n - \sum_{j =1}^\infty a_j^n \chi_{S_j^n} + (h-u_n)\chi_{S_0^n}$ converging in measure to a function $u\in GSBV^2(\Omega';\R^N)$, all still satisfying the boundary condition $v_n = h$ and $u = h$ on $\Omega'\setminus \Omega.$ The lower semi-continuity of the energy under convergence in measure (see \cite{ambrosioNewClass90,ambrosioLSCquasi94}) shows that $u$ is a minimizer of \eqref{eqn:nonlinEnergy}. 
We note that our {compactness result can} also account for heterogeneous bulk energies or (non-degenerate) cohesive surface energies, but do not spell out these details as this is already treated by Friedrich in \cite{friedrich_GSBVp} and might obfuscate the simplicity of our approach.

We remark that Dal Maso and Toader \cite{dalMasoToader_22} adapted Friedrich's strategy {to} the function space $GBV_*(\Omega')$ for the analysis of energies with plasticity, which in particular have degenerate cohesive surface energies. In \cite{donati_GBVstar_24}, Donati extended this to the vectorial setting of $GBV_*(\Omega';\R^d)$.
While highly non-trivial, in dimension $N=2,$ Friedrich's strategy does generalize to the case of linear elasticity, where the natural control is given in terms of the Griffith energy \cite{FriedrichSolombrino18}. Alternative approaches to compactness in linearized elasticity that work in any dimension were subsequently developed by Chambolle and Crismale \cite{chambolleCrismale18,chambolleCrismaleEquilibrium}  and Almi and Tasso \cite{AlmiTasso21}.  We remark that the techniques of \cite{chambolleCrismale18,chambolleCrismaleEquilibrium,AlmiTasso21} would also work in our setting, but our intention is to introduce a new approach, with compelling intuition, which may generate future insights and applications.

\noindent \textbf{Outline.} In Section \ref{sec:CC}, we prove a variant of the classical concentration-compactness theorem for infinitely many bubbles of concentration.  With this, in Section \ref{sec:mainProof}, we prove the compactness Theorem \ref{thm:main}.

\noindent \textbf{Notation.} Throughout we use $C>0$ as a generic constant, possibly changing from line to line; we make its dependencies explicit when necessary. The set $B(a,r)\subset \R^d$ is the open ball centered at $a\in \R^d$ with radius $r>0$. The set $\Omega'\subset \R^N$ refers to a bounded open set with Lipschitz boundary. For a set of finite perimeter $E\subset \Omega'$, we use $\partial^* E$ to denote the reduced boundary of $E$ in $\Omega'$. When we wish to refer to the reduced boundary of $E$ as a subset of $\R^N$, we write $\partial^*_{\R^N} E$. The function space $GSBV^p(\Omega';\R^d)$ consists of functions $u \in GSBV(\Omega';\R^d)$ such that
$\int_{\Omega'} |\nabla u|^p\, dx +\mathcal{H}^{N-1}(J_u) < +\infty$, where $GSBV(\Omega';\R^d)$ is as in \cite{AFP}, $\nabla u$ is the weak approximate gradient, and $J_u$ is the weak approximate jump-set. We emphasize that we always consider the jump-set $J_u$ as a subset of $\Omega'$. As usual, we write $GSBV^p(\Omega')$ for $GSBV^p(\Omega';\R)$.

\section{Mathematical preliminaries}\label{sec:CC}

The main tool for our proof is Lions' concentration-compactness theorem \cite{lions_CC_84}. 
\begin{theorem}\label{thm:CC}
Suppose $f_n : \R^d \to [0,\infty)$ and $\lambda>0$ with $\int_{\R^d} f_n \, dx \to \lambda$ as $n\to \infty.$ Then up to a subsequence (not relabeled) one of the three options holds:
\begin{enumerate}
\item \label{CC:compactness} \emph{(Compactness).} There are vectors $a^n_1\in \R^d$ such that for all $\eps>0$ there exists $R_\eps>0$ with $$\liminf_{n\to \infty} \int_{B(a^n_1,R_\eps)} f_n \, dx \geq \lambda - \eps.$$
\item \emph{(Vanishing).} For all $R>0$, $$\limsup_{n\to \infty }\left(\sup_{a\in \R^d}\int_{B(a,R)} f_n \, dx\right) = 0.$$
\item \emph{(Dichotomy).} There are $\lambda_1\in (0,\lambda)$, vectors $a^n_1\in \R^d$, and radii $0<R^n_1<\hat R^n_1$ with $\hat R^n_1 - R^n_1 \to \infty$ as $n\to \infty$ so that the functions $f_1^n : = f_n\chi_{B(a^n_1,R^1_n)} $ and $f^n_2:=f_n\chi_{B(a^n_1,\hat R^n_1)^c}$ satisfy
\begin{equation}\label{eqn:dichotomy}
\begin{aligned}
& \limsup_{n\to \infty}\left|\int_{\R^d} f_1^n\, dx - \lambda_1\right| = 0, \quad \limsup_{n\to \infty}\left|\int_{\R^d} f_2^n\, dx - (\lambda - \lambda_1)\right|= 0 , \\
& \quad\quad\quad\quad \text{and} \quad \limsup_{n\to \infty} \|f_n - (f_1^n+f_2^n)\|_{L^1(\R^d)} = 0. 
\end{aligned}
\end{equation}
Further $f_1^n$ satisfies the compactness of (\ref{CC:compactness}).
\end{enumerate}
\end{theorem}
We will actually rely on a version of this which accounts for multiple bubbles of concentration. The proof follows by repeatedly applying the Theorem \ref{thm:CC} to the leftover piece $f_2^n$ in the case of dichotomy.

\begin{corollary}\label{cor:bubbleDecomp}
Suppose $f_n : \R^d \to [0,\infty)$ and $\lambda>0$ with $\int_{\R^d} f_n \, dx \to \lambda$ as $n\to \infty.$ There is a subsequence of $f_n$ (not relabeled) along with positive numbers $(\lambda_{j})_{j=1}^{J^*}$, where $J^*\in \N\cup \{+\infty\}$, and vectors $a_j^n\in \R^d$, with $|a_j^n-a_i^n|\to \infty$ for $i \neq j$ as $n\to \infty,$ such that $\sum_j \lambda_j \leq \lambda$ and for any $\eps>0$, there is finite ${J} = J(\eps)\leq {J^*}$ and radius $R_\eps>0$ so that
 the remaining function ${f^{J}_n}:=(1-\sum_{j=1}^J\chi_{B(a_j^n,R_\eps)})f_n$ satisfies a weak form of vanishing:
\begin{equation}\label{eqn:weaklyVanishing}
\sup_{R>0}\limsup_{n\to \infty}\left(\sup_{a\in \R^d} \int_{B(a,R)}f^{J}_n\, dx\right)\leq \eps. 
\end{equation}
Furthermore, for any $\eps_0>0$ and $r>0$, there exists a radius $R_{\eps_0}$ (which can be taken larger than $R_\eps$) so that
\begin{equation}\label{eqn:cocnentrationBubbles}
\limsup_{n\to \infty} \left|\int_{B(a_j^n,R)} f_n \, dx - \lambda_j\right| < \eps_0 \quad \text{ for }j\leq J
\end{equation}
for $R \in [R_{\eps_0},R_{\eps_0}+r].$
\end{corollary}

\begin{proof}
We proceed iteratively. We first apply Theorem \ref{thm:CC} to the sequence $f_n$. If we are in the case of compactness or vanishing, we are done. Supposing we are in the case of dichotomy, the relation \eqref{eqn:dichotomy} shows that $f_1^n$ has a `center of mass' $a_1^n\in \R^d$ with mass $\lambda_1 \in (0,\lambda)$. Applying Theorem \ref{thm:CC} to the remaining sequence $f_2^n$, we either conclude with compactness or vanishing, or we are in the case of dichotomy and are left with a remaining function $f^n_3$ (playing the role of $f^n_2$ before).
Supposing the case of dichotomy, we find another mass $\lambda_2 \in (0,\lambda - \lambda_1)$ and center of mass $a_2^n \in \R^d$.  Due to the fact that $\hat R^n_1 - R_1^n\to \infty$, we see that $|a_2^n - a_1^n|\to \infty$. Repeating this procedure, we find masses $\lambda_{j} \in  (0,\lambda - \sum_{k=1}^{j-1}\lambda_k]$, vectors $a_j^n$ satisfying the relation $|a_j^n - a_i^n| \to \infty$ as $n\to \infty$ for $i\neq j$,  and radii $R^n_j < \hat R^n_j$,  where $j$ runs from $1$ to $J^*$ (possibly $J^* = \infty$). 
In fact, Theorem \ref{thm:CC} always chooses the largest mass so that $\lambda_{j+1}\leq \lambda_j.$
Precisely, denoting by $f_{j+1}^n$ the function left over in the $j$th step in the case of dichotomy, we have
\begin{equation}\label{eqn:massViaMaximal}
\lambda_j  = \lim_{R\to \infty} \limsup_{n\to \infty}\left(\sup_{a\in \R^N} \int_{B(a,R)}f^n_j \, dx\right) 
\end{equation}
(this follows from the proof of Theorem \ref{thm:CC}, see \cite{lions_CC_84}).

We now prove \eqref{eqn:weaklyVanishing}. For $\eps>0$ fixed, choose finite $J\leq J^*$ so that $\sum_{j=J+1}^{J^{*}}\lambda_j <\eps/2.$ We then choose $R_\eps>0$ so that 
\begin{equation}\nonumber
\limsup_{n\to \infty} \left|\int_{B(a_j^n,R_\eps)} f_n \, dx - \lambda_j\right| < \frac{\eps}{2J} \quad \text{ for }j\leq J.
\end{equation}
By \eqref{eqn:dichotomy} this implies that
\begin{equation}\label{eqn:annuliExcess}
\limsup_{n\to \infty} \int_{B(a_j^n,\hat R_j^n) \setminus B(a_j^n,R_\eps)} f_n \, dx  < \frac{\eps}{2J} \quad \text{ for }j\leq J.
\end{equation}
We suppose by contradiction that there is $R>0$ such that 
\begin{equation}\label{eqn:contraHyp}
\limsup_{n\to \infty}\left(\sup_{a\in \R^d} \int_{B(a,R)}{f^{J}_n}\, dx\right) > \eps.
\end{equation}
Note that 
\begin{equation}\nonumber
{f^J_n}  \leq f^n_{J+1} + \sum_{j=1}^{J} \chi_{B(a_j^n,\hat R_j^n) \setminus B(a_j^n,R_\eps)}f_n,
\end{equation}
so we can estimate using \eqref{eqn:annuliExcess} and \eqref{eqn:contraHyp} that
\begin{equation}\nonumber
\eps < \limsup_{n\to \infty}\left(\sup_{a\in \R^d} \int_{B(a,R)} f^{n}_{J+1} \, dx\right) + \eps/2.
\end{equation}
This implies $\limsup_{n\to \infty}\left(\sup_{a\in \R^d} \int_{B(a,R)} f^{n}_{J+1} \, dx\right) >\eps/2$, which by  \eqref{eqn:massViaMaximal} implies that $\lambda_{J+1} > \eps/2$, a contradiction to the choice of $J.$

Finally, \eqref{eqn:cocnentrationBubbles} follows from the fact that the functions $f_n \chi_{B(a_j^n,R_j^n)}$ are concentrated about $a_j^n$ with $\int_{\R^d}f_n \chi_{B(a_j^n,R_j^n)}\, dx  \to \lambda_j$ as $n\to \infty.$
\end{proof}

\section{Proof for $GSBV^p$ compactness}\label{sec:mainProof}

Our proof of the compactness Theorem \ref{thm:main} proceeds by applying the concentration-compactness Corollary \ref{cor:bubbleDecomp} to well-chosen functions defined on the range of the displacements $u_n$. Taking the pre-image of these bubble sets, we will slice our reference configuration $\Omega'$ into regions that correspond to concentrated clusters of `material' in the range $\R^d$. For fixed $\eps>0$ coming from Corollary \ref{cor:bubbleDecomp}, we pass to the limit $n\to \infty$ for an $\eps$-dependent modification of $u_n$ to construct an intermediate function $u_\eps$ (that does not depend on $n$). A subsequent diagonalization as $\eps\to 0$ and, simultaneously, $J\to J^*$ concludes the theorem. Additionally, we will show that vanishing cannot occur unless it corresponds to a vanishingly small region in the domain $\Omega'$; see Remark \ref{rmk:vanishExample} and Lemma \ref{lem:vanishing}.

We first note that since $GSBV^p(\Omega';\R^d) = [GSBV^p(\Omega')]^d$ by \cite[Prop. 2.3]{DMFraToa_Preprint} (see also \cite{DMFraToa05}) it will suffice to prove Theorem \ref{thm:main} in the case that $d=1.$ The only conclusion of Theorem~\ref{thm:main} that is not immediately satisfied by ensuring the analogous result holds for the components is the lower semi-continuity of the jump-set. This is because $J_{u} = \cup_{i=1}^d J_{u_i}$, where $u_i$ refers to the $i$th component of $u$, and we must be careful not to double count the jump-set when adding up the result for the components. To avoid this, one can apply \cite[Lemma A.4]{friedrichSteinkeStinson} {(which is stated for $N = 2$, but holds for all $N\geq 1$)} to find disjoint open Lipschitz sets $\mathcal{U}_i$ such that
\begin{equation}\nonumber
\mathcal{H}^{N-1}(J_u) - \eta \leq \sum_{i=1}^d \mathcal{H}^{N-1}(J_{u_i} \cap \mathcal{U}_i).
\end{equation}
Consequently, it suffices to prove lower semi-continuity of the component jump-sets on each $\mathcal{U}_i$ and then take $\eta\to 0$. As our lower semi-continuity argument will apply to any open set contained in $\Omega'$, from now on, we consider scalar-valued $u_n$, i.e., $d=1$.

For each $u_n\in GSBV^p(\Omega')$, we 
introduce the associated concentration function
\begin{equation}\label{eqn:fn}
f_n(t) : =f_n(t,\Omega'):=\mathcal{H}^{N-1}(\partial^*\{u_n>t\}\setminus J_{u_n}) + \sum_{\pm} \mathcal{H}^{N-1}(\{t-1<u_n^{\pm}<t+1\}\cap (J_{u_n}\cup \partial \Omega')),
\end{equation}
where $u^{\pm}$ denotes the two possible trace values of $u$ on the rectifiable set $J_{u_n}\cup \partial \Omega'.$ 

\begin{remark}\label{rmk:fnSetoFinitePerim}
We note that $f_n(t,\Omega')$ can be naturally defined for any $\Omega\subset \Omega'$ that is a set of finite perimeter. Precisely, taking $\Omega$ to be the  $\mathcal{L}^{N}$-equivalent  set such that every point of $x\in \Omega$ has density $1$, we let 
\begin{equation}\nonumber
f_n(t,\Omega):=\mathcal{H}^{N-1}(\left(\partial^*\{u_n>t\}\setminus J_{u_n}\right)\cap \Omega) + \sum_{\pm} \mathcal{H}^{N-1}(\{t-1<(u_n|_{\Omega})^{\pm}<t+1\}\cap (J_{u_n}\cup \partial^*_{\R^N} \Omega)).
\end{equation}
In the second term, we abuse notation a bit to emphasize that we only count the trace of $u_n$ from inside $\Omega$. It follows that $f_n$ is entirely determined by the values of $u_n$ in $\Omega.$
\end{remark}

The function $f_n$ will provide us with a convenient way to break up the range of $u_n\in GSBV^p(\Omega')$. Before jumping into the proof of Theorem \ref{thm:main}, we give an example showing how concentration-compactness sees $u_n$ through the functions $f_n$.

\begin{remark}[{Example with vanishing}]\label{rmk:vanishExample}
Let $\Omega =\Omega'  = (-1,1)\times (0,1)$. We define the sequence of functions $u_n \in GSBV^2(\Omega')$ by 
\begin{equation}\nonumber
u_n(x ): = \begin{cases}
0 & \text{ if }x \in (-1,0)\times (0,1) \\
i & \text{ if }x \in (0,\frac{1}{n})\times (\frac{i-1}{n},\frac{i}{n}) \\
n+1 & \text{ if }x \in (\frac{1}{n},1)\times (0,1) \\
\end{cases}
\end{equation}
In this case, $\mathcal{H}^1(J_{u_n}) = 3 - 1/n$. If we apply Corollary \ref{cor:bubbleDecomp} to the function $f_n$ from \eqref{eqn:fn}, we will find that there are two bubbles of concentration centered at $a_1^n = 0$ and $a_2^n = n+1$. Now, if we excise bubbles at these centers of mass to define ${f_n^2} : = f_n (1-\chi_{B(a_1^n,1)}-\chi_{B(a_2^n,1)})$ {(using again the notation $f_n^J$ from Corollary~\ref{cor:bubbleDecomp} with $J=2$)}, we still find that
$$\int_{\R}{f_n^2} \, dx \geq \mathcal{H}^{1}( \{1 \leq  u^{\pm} \leq n\}\cap  {(J_{u_n}\cup \partial \Omega')}) \geq 3 - \frac{1}{n},$$
but simultaneously, for any fixed radius $R>0$ we have
$$\sup_{a\in \R}\int_{B(a,R)}{f_n^2} \, dx \leq  \frac{CR}{n}.$$
Thus, $u_n$ is an example of a finite energy sequence, as in \eqref{eqn:energyBound}, where vanishing occurs, as in \eqref{eqn:weaklyVanishing}, without the total energy of the remaining function $f_n^J$ also vanishing.
\end{remark}

We show that the above example is essentially the canonical example in the case of vanishing. Precisely, if $f_n$ vanishes on a given set, then the pre-image of that set under $u_n$ must have vanishing volume; in Remark \ref{rmk:vanishExample}, this is the set $(0,\frac{1}{n})\times (0,1)$. This idea is quantified in the following lemma, whose proof is deferred to after the proof of the main Theorem \ref{thm:main}.

\begin{lemma}\label{lem:vanishing}
Let the hypotheses of Theorem \ref{thm:main} hold with $d=1$. Additionally, let $\Omega_n\subset \Omega$ be sets of finite perimeter with uniformly bounded perimeter such that the sequence $f_n(t,\Omega_n)$, defined in Remark \ref{rmk:fnSetoFinitePerim}, is weakly vanishing in the sense that
\begin{equation}\label{eqn:weakVanLemma}
\sup_{R>0}\limsup_{n\to \infty}\left(\sup_{a\in \R} \int_{B(a,R)}f_n(t,\Omega_n)\, dt\right)\leq \eps
\end{equation}
for some $\eps>0.$
Then it follows that 
$$\limsup_{n\to \infty} \mathcal{L}^N(\Omega_n)\leq C\eps^{\frac{1}{N-1}} ,$$
where $C>0$ only depends on the uniform bound on the perimeters of $\Omega_n$ and the uniform energy bound from \eqref{eqn:energyBound}.
\end{lemma}

With these preliminaries in place, we can now prove the main compactness theorem.

\begin{proof}[Proof of Theorem \ref{thm:main}]
As discussed, we may assume $d=1$. Furthermore, if the theorem holds for $v_n := u_n - h$, which has $v_n = 0$ on $\Omega'\setminus \Omega$, with limit $v \in GSBV^p(\Omega')$, sets $\mathcal{S}^n$, and constants $a_j^n\in \R$, the theorem holds for the sequence $u_n$ with limit $u:= v+h$ and the same collection of sets and constants. Thus, we may assume that $h = 0.$

Let $f_n$ be defined as in \eqref{eqn:fn}. Since $$\mathcal{H}^{N-1}(J_{u_n}\cup \partial \Omega')\leq \int_{\R}f_n \, dx \leq C\left(\|\nabla u_n\|_{L^p(\Omega)} + \mathcal{H}^{N-1}(J_{u_n}\cup \partial \Omega')\right)$$ by the coarea formula and the inequality
\begin{equation}\label{eqn:fnCrackRel}
\sum_{z\in a +\mathbb{Z}} \chi_{[z,z+1)}(t)\mathcal{H}^{N-1}(\{z\leq u^{\pm}<z+1\}\cap (J_{u_n}\cup \partial \Omega')) \leq \mathcal{H}^{N-1}(\{t-1<u^{\pm}<t+1\}\cap (J_{u_n}\cup \partial \Omega')),
\end{equation}
which holds for any $a\in \R$, we may assume up to a subsequence (not relabeled) that there is $\lambda>0$ such that $\int_{\R}f_n\, dx \to \lambda$. We apply Corollary \ref{cor:bubbleDecomp} to the sequence $f_n$ to find $J^*\in \N\cup \{+\infty\}$ and constants $a_j^n \in \R$, such that for $0<\eps<1$ fixed there is finite ${J} = J(\eps)\leq J_*$ and $R_{\eps}>0$ satisfying \eqref{eqn:weaklyVanishing};  for $r = 2$, we also set $\eps_0 : = \eps/J$ to find $R_{\eps_0}$.

\textit{Step 1 (Construction of partition).}
Fix $0<\eps<1$. For notational convenience within this step, suppose $a_j^n < a_{j+1}^n$ for $j = 1,\ldots, {J}$. We partition the domain $\Omega'$ into a variety of ($\eps$-dependent) sets. See Figure~\ref{fig:partition} for a depiction of the below construction in the range. Precisely, for $R_j^n \in [R_{\eps_0},R_{\eps_0} +1)$, we introduce the \emph{main partition elements}
\begin{equation}\label{eqn:Pj}
P_j^n : = \{x\in \Omega' : -R_j^n \leq u_n(x) - a_j^n< R_j^n\}.
\end{equation}
We also introduce the \emph{gap sets} given by
\begin{equation}\label{eqn:Gjpm}
\begin{aligned}
G^n_{j,+} & : = \{x\in \Omega': R_j^n \leq u_n(x) - a_j^n< R_j^n +1 \}, \\
G^n_{j,-}  & : = \{x\in \Omega': -R_j^n - 1 \leq u_n(x) - a_j^n< -R_j^n \}.
\end{aligned}
\end{equation}
Finally we introduce the \emph{vanishing sets}
\begin{equation}\label{eqn:Vj}
V_j^n : = \{x\in \Omega' : a_j^n + R_j^n + 1 \leq u_n(x) < a_{j+1}^n -R_{j+1}^n - 1\},
\end{equation}
where $j \in \{0,\ldots,J\}$ and, by an abuse of notation, we take $a_0^n = -\infty$ and $a_{J+1}^n = +\infty.$
Note that the above sets form a disjoint partition of $\Omega'$ (perhaps excluding a set of measure zero).

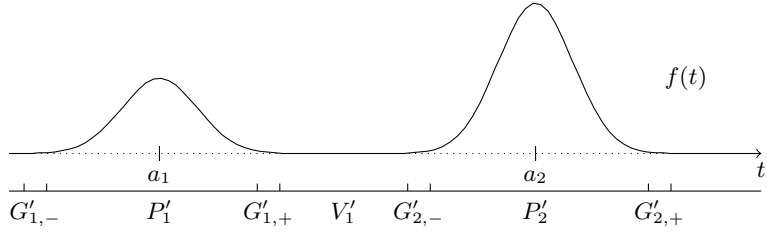
\begin{figure}
\begin{tikzpicture}


    \draw[dotted,->] (-5,0) -- (5,0);
    \node[below] at (5,0) {\small $t$};

    \draw (-5.00,0.00)--(-4.90,0.00)--(-4.80,0.00)--(-4.70,0.00)--(-4.60,0.01)--(-4.49,0.01)--(-4.39,0.02)--(-4.29,0.04)--(-4.19,0.06)--(-4.09,0.09)--(-3.99,0.14)--(-3.89,0.21)--(-3.79,0.29)--(-3.69,0.39)--(-3.59,0.50)--(-3.48,0.62)--(-3.38,0.74)--(-3.28,0.85)--(-3.18,0.94)--(-3.08,0.99)--(-2.98,1.00)--(-2.88,0.97)--(-2.78,0.91)--(-2.68,0.81)--(-2.58,0.70)--(-2.47,0.58)--(-2.37,0.46)--(-2.27,0.35)--(-2.17,0.25)--(-2.07,0.18)--(-1.97,0.12)--(-1.87,0.08)--(-1.77,0.05)--(-1.67,0.03)--(-1.57,0.02)--(-1.46,0.01)--(-1.36,0.00)--(-1.26,0.00)--(-1.16,0.00)--(-1.06,0.00)--(-0.96,0.00)--(-0.86,0.00)--(-0.76,0.00)--(-0.66,0.00)--(-0.56,0.00)--(-0.45,0.00)--(-0.35,0.00)--(-0.25,0.00)--(-0.15,0.00)--(-0.05,0.00)--(0.05,0.00)--(0.15,0.00)--(0.25,0.00)--(0.35,0.01)--(0.45,0.02)--(0.56,0.03)--(0.66,0.05)--(0.76,0.09)--(0.86,0.15)--(0.96,0.23)--(1.06,0.34)--(1.16,0.49)--(1.26,0.67)--(1.36,0.89)--(1.46,1.13)--(1.57,1.37)--(1.67,1.60)--(1.77,1.80)--(1.87,1.93)--(1.97,2.00)--(2.07,1.98)--(2.17,1.89)--(2.27,1.72)--(2.37,1.51)--(2.47,1.27)--(2.58,1.03)--(2.68,0.80)--(2.78,0.60)--(2.88,0.43)--(2.98,0.29)--(3.08,0.19)--(3.18,0.12)--(3.28,0.07)--(3.38,0.04)--(3.48,0.02)--(3.59,0.01)--(3.69,0.01)--(3.79,0.00)--(3.89,0.00)--(3.99,0.00)--(4.09,0.00)--(4.19,0.00)--(4.29,0.00)--(4.39,0.00)--(4.49,0.00)--(4.60,0.00)--(4.70,0.00)--(4.80,0.00)--(4.90,0.00)--(5.00,0.00);
    \node at (4,1) {\small $f(t)$};

    \draw (-3,.1)--(-3,-.1);
    \node[below] at (-3,-.1) {\small $a_1$};

    \draw (2,.1)--(2,-.1);
    \node[below] at (2,-.1) {\small $a_2$};

    \begin{scope}[shift = {(0,-.5)}]
    \draw (-5,0) -- (5,0);
    
    \draw (-4.8,.1)--(-4.8,0);
    \draw (-4.5,.1)--(-4.5,0);
    \node[below] at (-4.65,0) {\small $G_{1,-}'$};
    \node[below] at (-3,0) {\small $P_1'$};
    \draw (-1.7,.1)--(-1.7,0);
    \draw (-1.4,.1)--(-1.4,0);
    \node[below] at (-1.55,0) {\small $G_{1,+}'$};
    \node[below] at (-.55,0) {\small $V_1'$};
    \draw (.3,.1)--(.3,0);
    \draw (.6,.1)--(.6,0);
    \node[below] at (.45,0) {\small $G_{2,-}'$};
    \node[below] at (2,0) {\small $P_2'$};
    \draw (3.5,.1)--(3.5,0);
    \draw (3.8,.1)--(3.8,0);
    \node[below] at (3.65,0) {\small $G_{2,+}'$};

    \end{scope}
\end{tikzpicture}
\caption{Partitioning the range via the concentration{-}compactness bubbles. In the figure sets $E'$ denote $u(E)$, i.e. $P_1' = u(P_1)$ etc.}
\label{fig:partition}
\end{figure}

We will choose $R_j^n \in [R_{\eps_0},R_{\eps_0} +1)$ such that the above sets form a Caccioppoli partition of $\Omega'$ and the gap sets have small boundary.
By \cite[Theorem 4.34]{AFP}, almost every choice of $R_j^n$ gives that the sets in \eqref{eqn:Pj}, \eqref{eqn:Gjpm}, and \eqref{eqn:Vj} are sets of finite perimeter. Since there are only finitely many sets, this implies that $P_j^n$, $G^n_{j,\pm}$, and $V_j^n$ form a Caccioppoli partition of $\Omega'$ for sufficiently large $n$.
 To control the boundary of the gap sets, we look at $G_{j,+}^n$.
Note from \eqref{eqn:cocnentrationBubbles}, we have that $\int_{a_{j}^n +R_{\eps_0}}^{a_{j}^n + R_{\eps_0} + 1} \left(f_n(t) + f_n(t+1)\right) \, dt \leq 2\eps_0$ for sufficiently large $n.$
 By an averaging argument applied to this inequality, we can additionally choose $R_j^n \in [R_{\eps_0},R_{\eps_0} +1)$ so that
\begin{equation}\label{eqn:ineq1}
\begin{aligned}
\mathcal{H}^{N-1}(\partial^*\{u_n>a_j^n + R_j^n\}\setminus J_{u_n}) + &\mathcal{H}^{N-1}(\partial^*\{u_n>a_j^n + R_j^n + 1\}\setminus J_{u_n}) \\
&\quad \leq f_n(a_j^n + R_j^n) + f_n(a_j^n + R_j^n+ 1) \leq C \eps_0,
\end{aligned}
\end{equation} 
for a constant $C>0$, and $\mathcal{L}^N(\{u_n = a_j^n + R_j^n\}) + \mathcal{L}^N(\{u_n = a_j^n + R_j^n+1\}) = 0$. Furthermore \eqref{eqn:fnCrackRel} with $a = a_j^n + R_{\eps_0} $ and smallness of the integral of $f_n$ over $[R_{\eps_0},R_{\eps_0}+2]$ implies that 
\begin{equation}\label{eqn:ineq2}
\mathcal{H}^{N-1}\left(\{ R_{\eps_0}\leq u^{\pm}- a_j^n<R_{\eps_0}+2\}\cap (J_{u_n}\cup \partial \Omega')\right)\leq 2\eps_0.
\end{equation}
Having chosen $R_j^n$ so that $G^{n}_{j,+}$ is a set of finite perimeter, we see 
\begin{equation}\nonumber
\begin{aligned}
\partial^*_{\R^N}G_{j,+}^n \subset & \bigcup_{\pm}\left(\{ R_{\eps_0}\leq u^{\pm}-a_j^n<R_{\eps_0}+2\}\cap (J_{u_n}\cup \partial \Omega')\right)  \\
& \quad \cup  \left(\partial^*\{u_n>a_j^n + R_j^n\}\setminus J_{u_n}\right) \cup \left(\partial^*\{u_n>a_j^n + R_j^n+1\}\setminus J_{u_n} \right),
\end{aligned}
\end{equation}
which shows by \eqref{eqn:ineq1} and \eqref{eqn:ineq2} that $\mathcal{H}^{N-1}(\partial^*_{\R^N}G_{j,+}^n) \leq C\eps_0$. Applying this reasoning for all the gap sets, taking the union over the sets, and recalling that $\eps_0 : = \eps/J$, we have 
\begin{equation}\label{eqn:gapLength}
\mathcal{H}^{N-1}\Big(\bigcup_{j,\pm}\partial^*_{\R^N}G_{j,\pm}^n\Big)\leq C\eps
\end{equation}
for a universal constant $C>0$. (We note, though we only implicitly use it for $G_{j,+}^n$ and $G_{j,-}^n$, that one could sum before doing the averaging argument so that $R_j^n = R_i^n$ for all $i$ and $j$.)
From the isoperimetric inequality in $\R^{N}$, we also find that 
\begin{equation}\label{eqn:gapVolume}
\mathcal{L}^N \Big(\bigcup_{j,\pm}G_{j,\pm}^n\Big) \leq C \eps^{\frac{N}{N-1}}.
\end{equation}
Inequality \eqref{eqn:gapLength} shows us that we can add the boundary of the gap sets to the crack while staying close to the original energy. And further \eqref{eqn:gapVolume} shows that the value of the deformation/displacement $u_n$ on this set will not affect the limit (once $\eps\to 0$). Consequently, we are free to prescribe $u_n$ on $\cup_{j,\pm}G_{j,\pm}^n$.

Our next objective is to control the measure of the vanishing region. Since $P_j^n$, $G_{j,\pm}^n$, and $V_{j}^n$ form a Caccioppoli partition of the domain, by \cite[Theorem 4.17]{AFP} it holds that
\begin{equation}\label{eqn:CaccProp}
\partial^* P_i^n \subset \Big(\bigcup_{j\neq i} \partial^* P_{j}^n\Big) \cup\Big(\bigcup_{j,\pm}\partial^* G_{j,\pm}^n\Big) \cup \Big(\bigcup_{j} \partial^* V_{j}^n\Big),
\end{equation}
with each point of $\partial^* P_i^n$ belonging to at most one set on the right-hand side, up to an $\mathcal{H}^{N-1}$-null set. Due to the gap sets, $\partial^* P_i^n \cap \partial^* P_{j}^n \subset J_{u_n}$ and $\partial^* P_i^n \cap \partial^* V_{j}^n \subset J_{u_n}$ ($i\neq j$ only in the first relation).
{We remark that these subset relations follow from the definition of the upper and lower limits used to define the jump-set of $GSBV$ functions, see \cite[Definition 4.28]{AFP}, and the fact that at every point $x\in \partial^*P_j^n$, $B(0,1)\cap \frac{1}{r}(P_j^n-x)$ blows up to a half ball as $r\to \infty$. For instance, if $x\in \partial^* P_i^n \cap \partial^* P_{j}^n$ with $ a_i^n<a_j^n$, then (outside of an $\mathcal{H}^{N-1}$-null set for $x$)
$$u^+(x)  = \inf\Big\{t \in \R : \lim_{r\to 0} \frac{\mathcal{L}^N(\{u>t\}\cap B(x,r))}{r^N}=0\Big\}\geq a_j^n - R_j^n $$
and similarly $u^-(x)\leq a_i^n + R_i^n$. Since $u^+(x)> u^-(x)$, $x\in J_u$ as desired.}
 As the analogous {subset} relations hold for each vanishing set, we use the above relations with \eqref{eqn:gapLength} to find
\begin{equation}\label{eqn:partitionOutsideCrack}
\mathcal{H}^{N-1}\Big( \Big[ \Big(\bigcup_{j} \partial^* P_{j}^n\Big)\cup \Big(\bigcup_{j} \partial^* V_{j}^n\Big) \Big]\setminus J_{u_n}\Big) \leq 
\mathcal{H}^{N-1}\Big(\bigcup_{j,\pm}\partial^* G_{j,\pm}^n\Big)\leq C\eps.
\end{equation}
In contrast to the gap sets, it is possible that the length of the boundary of the vanishing region is large. However, Lemma \ref{lem:vanishing} shows that they vanish in measure: Taking $\Omega_n : = \cup_j V_j^n$, \eqref{eqn:partitionOutsideCrack} implies $\mathcal{H}^{N-1}(\partial^*_{\R^{N}}\Omega_n)\leq \mathcal{H}^{N-1}(J_{u_n} \cup \partial \Omega') + C\eps.$ For $B(a,R) \subset \R$, we have 
\[\limsup_{n\to \infty} \int_{B(a,R)}f_n(t,\Omega_n)\, dt \leq  \limsup_{n\to \infty}\int_{B(a,R)}{f_n^J}(t) \, dt +C\eps\leq C\eps,\]  
by \eqref{eqn:weaklyVanishing}, \eqref{eqn:Vj}, and \eqref{eqn:partitionOutsideCrack} (see also Remark \ref{rmk:fnSetoFinitePerim}).
It follows that the hypotheses of Lemma \ref{lem:vanishing} are satisfied, and we conclude that 
\begin{equation}\label{eqn:VnsetsVolume}
\mathcal{L}^N (\cup_j V_j^n)\leq C\eps^{\frac{1}{N-1}} \quad \text{ for sufficiently large $n$},
\end{equation} where $C>0$ depends only on $\Omega'$ and the energy bound~\eqref{eqn:energyBound}.

 {We summarize that we have now shown that $\Omega'$ can be divided into the main partition pieces $P_j^n$, where $u_n$ is close $a_j^n$, and a remaining region that vanishes according to \eqref{eqn:gapVolume} and \eqref{eqn:VnsetsVolume}.}

\textit{Step 2 (Passing to the limit).}
To pass to the limit, we introduce a sequence $\eps_k\downarrow 0$ as $k\to \infty$. By an abuse of notation, we just denote this sequence by $\eps$. We apply Step 1 to the sequence $u_n$ for each fixed $\eps$. We will use the ($\eps$-dependent) partition to construct modified functions $u_{\eps,n}$, similar to those in the theorem statement. After passing $n\to \infty$, we will let $\eps \to 0$ and use a diagonalization argument to conclude.

\textit{Substep 2.1 ($n\to \infty$).}
Now with $\eps>0$ fixed, we will pass to the limit as $n \to \infty$ using classical compactness results for $SBV^p({\Omega'})$ functions. We first {apply Step 1 (with $\eps>0$ fixed) to} define the Caccioppoli partition $\mathcal{P}_\eps^n : = \{P_j^n\}_{j=1}^J \cup\{V_\eps^n\}$, where $V_\eps^n : = \Omega' \setminus (\cup_{j=1}^J P_j^n)$ is the collection of all gap and vanishing sets. By \eqref{eqn:gapVolume} and \eqref{eqn:VnsetsVolume}, 
\begin{equation}\label{eqn:VepssetsVolume}
\mathcal{L}^N (V_\eps^n)\leq C\eps^{\frac{1}{N-1}} \quad \text{ for sufficiently large $n$},
\end{equation}
 With this, we define our $\eps$-modification of the sequence $u_n$ as
\begin{equation}\label{eqn:uepsn}
u_{\eps,n}(x) : = u_n(x) - \sum_{j=1}^J a_j^n \chi_{P_j^n}(x) + (0-u_n(x))\chi_{V_\eps^n}(x) = \begin{cases}
u_n(x)  - a_j^n & \text{ for }x \in P_j^n, \\
0 & \text{ otherwise}.
\end{cases}
\end{equation}
Since $\|u_{\eps,n}\|_{L^\infty(\Omega')} \leq R_{\eps_0}+2$, $\mathcal{H}^{N-1}(J_{u_{\eps,n}}) \leq \mathcal{H}^{N-1}(J_{u_n})+C\eps$ by \eqref{eqn:partitionOutsideCrack}, and \eqref{eqn:energyBound} holds, we apply compactness for $SBV^p(\Omega')$ functions with $L^\infty$-bounds \cite[Theorems 4.7 and 4.8]{AFP} to find a function $u_\eps \in SBV^p(\Omega')$ so that as $n\to \infty$ it holds that $u_{\eps,n} \to u_\eps $ in measure on $\Omega'$, $\nabla u_{\eps,n} \weakly \nabla u_{\eps}$ in $L^p(\Omega';\R^N)$, and 
\begin{equation}\label{eqn:jumpLSCn}
\mathcal{H}^{N-1}(J_{u_\eps}) \leq \liminf_{n\to \infty} \mathcal{H}^{N-1}(J_{u_{\eps,n}}) \leq \liminf_{n\to \infty} \mathcal{H}^{N-1}(J_{u_{n}}) +C\eps.
\end{equation}
For $\eps$ sufficiently small so that $\mathcal{L}^{N}(V_\eps^n)\leq C\eps^{\frac{1}{N-1}} < \mathcal{L}^{N}(\Omega'\setminus \Omega)$, there must be $j_0\in \{1,\ldots,J\}$ with $\Omega'\setminus \Omega \subset P_{j_0}^n$. The convergence of $u_{\eps,n}$ implies that $a_{j_0}^n$ converges as $n\to \infty$. Thus, without disturbing the convergence of $u_{\eps,n}$ to $u_\eps$, we may take $a_{j_0}^n = 0$ for all $n\in \N$, so that $u_{\eps,n} = u_\eps = 0$ on $\Omega'\setminus \Omega.$

For future use, we also note that by \eqref{eqn:energyBound}, \eqref{eqn:partitionOutsideCrack}, and compactness for sets of finite perimeter, the sets $V_\eps^n$ converges (up to a subsequence not relabeled) to a remaining set $V_\eps$ in $L^1(\Omega')$. In fact, $V_{\eps'} \subset V_{\eps}$ for $\eps' < \eps$ and $\mathcal{L}^N(V_\eps) \leq C\eps^{\frac{1}{N-1}}$, since the analogous relations hold for $V_\eps^n$ with sufficiently large $n$ (we may always assume $R_j^n = R_j^n(\eps) \leq R_j^n(\eps')$).

\textit{Substep 2.2 ($\eps \to 0$).}
Now we pass to the limit as $\eps\to 0$. 
Due to the subset relation for $V_\eps$, the function $u: \Omega'\to \R$ given by
\begin{equation}\nonumber
u(x) : = u_\eps(x) \quad  \text{ if }x\not\in V_\eps
\end{equation}
is well-defined. To see this, given $\eps > \eps' >0$, note that {for} $x \not\in V_{\eps}^n$, $u_{\eps,n}(x) = u_{{\eps'},n} (x)$. Since for a.e. $x \not \in V_\eps$ (up to a subsequence not relabeled), we have $u_{\eps,n}(x) \to u_\eps (x)$ and $u_{\eps',n}(x) \to u_{\eps'} (x)$ and, eventually, $x\not \in V_\eps^n$ for all sufficiently large $n$, we conclude that $u_\eps (x) = u_{\eps'}(x)$ for a.e. $x \not \in V_\eps$. Lastly, $\mathcal{L}^N(V_\eps)\to 0$ as $\eps\to 0$, so that a.e. $x \in \Omega'$ belongs to $(V_\eps)^c$ for sufficiently small $\eps.$ It follows that $u$ is defined for a.e. $x\in \Omega'.$

 We now prove that $u\in GSBV^p(\Omega')$ and as $\eps\to 0$ it holds that $u_{\eps} \to u $ in measure on $\Omega'$, $\nabla u_{\eps} \weakly \nabla u$ in $L^p(\Omega')$, and $\mathcal{H}^{N-1}(J_{u}) \leq \liminf_{\eps\to 0} \mathcal{H}^{N-1}(J_{u_{\eps}})$. The convergence in measure follows from the fact that $\mathcal{L}^N(V_\eps)\to 0$ as $\eps\to 0.$ With this, we note for any compactly supported function $\phi \in C^1_c(\R)$, $\phi\circ u_\eps$ belongs to $SBV^p(\Omega)$ with $\int_{\Omega'} |\nabla (\phi\circ u_\eps)|^p\, dx + \mathcal{H}^{N-1}(J_{\phi\circ u_\eps})$ uniformly bounded. By the compactness \cite[Theorems 4.7 and 4.8]{AFP} applied to $\phi\circ u_\eps$ and the convergence in measure of $u_\eps$, we see that $\phi\circ u \in SBV^p(\Omega')$, so that $u \in GSBV(\Omega')$ (here $p$ is only used to prevent a Cantor part from appearing). By \cite[Remark 4.32]{AFP}, at any point of density one with respect $(V_\eps)^c$, we have that $\nabla u = \nabla u_\eps$. This implies that $\nabla u_\eps$ converges to $\nabla u$ pointwise a.e. in $\Omega'$. Since $\nabla u_\eps$ are uniformly bounded in $L^p(\Omega';\R^N)$, there is $g \in L^p(\Omega';\R^N)$ such that $\nabla u_\eps \weakly g.$ Mazur's lemma shows that weak and pointwise limits must coincide, so that $\nabla u_\eps \weakly \nabla u$ in $L^p(\Omega';\R^N)$. Lower semi-continuity of the jump-set
 \begin{equation}\label{eqn:jumpLSCeps}
 \mathcal{H}^{N-1}(J_{u}) \leq \liminf_{\eps\to 0} \mathcal{H}^{N-1}(J_{u_{\eps}})
 \end{equation}
  can be shown using slicing; this is a standard, albeit technical, idea, so we defer the proof to Appendix~\ref{sec:appendix}.
 
\textit{Substep 2.3 (Diagonalization).} As convergence in measure and weak convergence in bounded subsets of $L^p(\Omega';\R^N)$ are metrizable, a diagonalization argument as $n\to \infty$ and then $\eps \to 0$ concludes the theorem. We point out that lower semi-continuity of the jump-set follows from \eqref{eqn:jumpLSCn} and \eqref{eqn:jumpLSCeps}. Further, $S_0^n$ and $(S_j^n)_{j=1}^\infty$ come from $V_\eps^n$ and $\{P_j^n\}_{j=1,j\neq j_0}^{J(\eps)}$, respectively, with the vanishing of $S^n_0$ a consequence of \eqref{eqn:VepssetsVolume} and the vanishing length of the sets $\mathcal{S}^n$ outside of the jump-set $J_{u_n}$ coming from \eqref{eqn:partitionOutsideCrack}. Lastly, $|a_j^n  - a_i^n|\to \infty$ for $i\neq j$ as $n\to \infty$ is guaranteed by Corollary \ref{cor:bubbleDecomp}.
 \end{proof}
 
 As we show in the following remark, the length of the partition can also be controlled by the crack energy.

 \begin{remark}[Accounting for the partition]
 In Step 2 of the Proof of Theorem \ref{thm:main} we introduced the partition $\mathcal{P}_\eps^n$. By \eqref{eqn:energyBound}, \eqref{eqn:CaccProp}, and \eqref{eqn:partitionOutsideCrack}, the length of the Caccioppoli partition $\mathcal{P}_\eps^n$ is uniformly bounded independent of $\eps$ and $n$, so that we may apply compactness for Caccioppoli partitions \cite[Theorem 4.19]{AFP} to find that $\mathcal{P}^n_\eps$ converges to a new partition $\mathcal{P}_\eps$ as $n\to \infty$. Similarly, the partitions $\mathcal{P}_\eps$ inherit the uniform bound on length from $\mathcal{P}_\eps^n$, so that as before, up to a subsequence, the partitions $\mathcal{P}_\eps$ converge to a limit Caccioppoli partition $\mathcal{P}$ as $\eps\to 0$.
 
 Note that by perturbing the constants $a_j^n$ in the definition of $u_{\eps,n}$ \eqref{eqn:uepsn} by well-chosen constants $\alpha_j^n \in [0,1]$, we can form a new function $\tilde u_{\eps,n }$ such that $\partial^* \mathcal{P}_\eps^n \subset J_{\tilde u_{\eps,n}}$ up to an $\mathcal{H}^{N-1}$-null set ($\partial^* \mathcal{P}$ is shorthand for the union of the reduced boundaries of the sets in the partition $\mathcal{P}$). Likewise we can make sure this relation passes to the limit $\partial^* \mathcal{P}_\eps \subset J_{\tilde u_{\eps}}$ as $n\to \infty$ and the next limit $\partial^* \mathcal{P} \subset J_{\tilde u}$ as $\eps\to 0.$
 Stitching together \eqref{eqn:jumpLSCn} and \eqref{eqn:jumpLSCeps} for the perturbed sequence, we find 
 $ \mathcal{H}^{N-1}(J_{\tilde u})\leq \liminf_{n\to \infty}\mathcal{H}^{N-1}(J_{u_{n}}).$
 But by the nature of the piecewise constant perturbation, $J_{\tilde u} = J_u \cup \partial^*\mathcal{P}$ so that
\[\mathcal{H}^{N-1}(J_{ u}\cup \partial^* \mathcal{P})\leq \liminf_{n\to \infty}\mathcal{H}^{N-1}(J_{u_{n}}).\]
\end{remark}

If one wished, analogous to the compactness \cite[Theorem 1.1]{chambolleCrismaleEquilibrium} of Chambolle and Crismale, it is direct from Theorem \ref{thm:main} and the previous remark to cast our compactness statement in terms of a single limit partition $\mathcal{P} = \{P_j\}_{j=0}^\infty$ and obtain convergence for the sequence $u_n - \sum_{j=1}^\infty a_{j}^n \chi_{P_j}$.
 
 We turn to proving the lemma regarding the vanishing case.

 \begin{proof}[Proof of Lemma \ref{lem:vanishing}]
Let $\eps>0$ be as in the Lemma statement and define $m_n:=\mathcal{L}^N(\Omega_n).$ We divide $\R$ into $\alpha\in \N$ intervals using the points $-\infty := t_0^n < t_1^n <\cdots < t_\alpha^n =:\infty$ such that 
\begin{equation}\label{eqn:betweenVolume}
\mathcal{L}^N (\{x\in \Omega_n : t_{i}^n \leq u_n < t_{i+1}^n  \})\geq \frac{m_n}{\alpha} - C\eps^{\frac{N}{N-1}} \quad \text{ for }i = 0,\ldots, \alpha-1,
\end{equation}
where $C>0$ is a constant independent of $\eps$ and $n$, which is possible since for any fixed $R>0$ and $n$ sufficiently large, we have for a.e. $t\in \R$
\begin{equation}\label{eqn:gapVolume2}
\mathcal{L}^N(\{x\in \Omega_n: t-R \leq  u_n <t+R\})\leq C \mathcal{H}^{N-1}(\partial^*_{\R^N}\{x\in \Omega_n: t-R \leq u_n <t+R\})^{\frac{N}{N-1}} \leq C \eps^{\frac{N}{N-1}},
\end{equation}
where we have used \eqref{eqn:weakVanLemma} and reasoning as before \eqref{eqn:gapLength}. We introduce \emph{gap sets} centered around each $t_i^n$ defined as
\begin{equation}\nonumber
G_i^n : = \{x\in \Omega_n : -R < u_n -t_i^n <R \} \quad  \text{ for }i = 1,\ldots, \alpha-1,
\end{equation}
and use \eqref{eqn:gapVolume2} to find $\mathcal{L}^N (G_i^n) \leq C \eps^{\frac{N}{N-1}};$ note since the left-hand side of \eqref{eqn:betweenVolume} is left-continuous as a function of $t_i^n$, each gap set may be taken to be a set of finite perimeter.
Similarly, we define the \emph{inbetween sets}
\begin{equation}\nonumber
B_{i}^n : = \{x\in \Omega_n : t_{i}^n+R \leq u_n < t_{i+1}^n-R  \} \quad \text{ for }i = 0,\ldots, \alpha-1.
\end{equation}
Since $G_i^n$ has small volume, we have from \eqref{eqn:betweenVolume} that
$\mathcal{L}^N (B_{i}^n)\geq \frac{m_n}{\alpha} - C\eps^{\frac{N}{N-1}}$ for a possibly bigger constant $C>0.$

Analogous to the reasoning for \eqref{eqn:CaccProp} and \eqref{eqn:partitionOutsideCrack}, we note that $G_i^n$ and $B_i^n$ form a Caccioppoli partition of $\Omega_n$, so that 
$$\partial^*_{\R^N} B_i^n \setminus (\partial^*G_i^n \cup \partial^*G_{i+1}^n)\subset J_{u_n}\cup \partial^*_{\R^N} \Omega_n,$$
with a point $x\in \partial^*_{\R^N} B_i^n$ belonging to at most one other set $\partial^*_{\R^N} B_j^n$. With this, we use \eqref{eqn:gapVolume2} for the surface area of $G_i^n$ and the isoperimetric inequality to estimate
\begin{equation}\nonumber
\begin{aligned}
\mathcal{H}^{N-1}(J_{u_n}\cup \partial^*_{\R^N} \Omega_n) & \geq  \frac{1}{2}\sum_{i}\mathcal{H}^{N-1}(\partial^*_{\R^N} B_i^n \setminus (\partial^*G_i^n \cup \partial^*G_{i+1}^n)) \\
& \geq \frac{1}{2}\sum_{i}\mathcal{H}^{N-1}(\partial^*_{\R^N} B_i^n ) - C\alpha \eps \\
& \geq \frac{1}{C}\sum_{i}\left(\frac{m_n}{\alpha} - C\eps^{\frac{N}{N-1}}\right)^{\frac{N-1}{N}} - C\alpha \eps \\
& \geq \frac{1}{C}\sum_i\left(\frac{m_n}{\alpha}\right)^{\frac{N-1}{N}} - C\alpha \eps = \frac{1}{C}m_n^{\frac{N-1}{N}}\alpha^{\frac{1}{N}} - C\alpha \eps,
\end{aligned}
\end{equation}
where $C>0$ is independent of $\eps$ and $n$.
Rearranging, we conclude that 
$$ m_n \leq C\left( \mathcal{H}^{N-1}(J_{u_n}\cup \partial^*_{\R^N} \Omega_n) +  C\alpha \eps\right)^{\frac{N}{N-1}} \alpha^{-\frac{1}{N}}.$$
Choosing $\alpha = \lceil \frac{1}{\eps} \rceil $ concludes the proof of the lemma.
\end{proof}

\appendix

\section{Lower semi-continuity of the jump-set}\label{sec:appendix}

We prove the following lower semi-continuity result for the jump-sets.

\begin{lemma}
Suppose that $u_n \in GSBV^p(\Omega')$ are such that 
\begin{equation}\label{eqn:energyBoundSlice}
\liminf_{n\to \infty} \left(\int_{\Omega'} |\nabla u_n|^p \, dx +\mathcal{H}^{N-1}(J_{u_n})\right) < \infty
\end{equation} and $u_n\to u \in GSBV(\Omega')$ in measure as $n\to \infty$. Then it holds that 
\begin{equation}\nonumber
\mathcal{H}^{N-1}(J_{u})\leq \liminf_{n\to \infty}\mathcal{H}^{N-1}(J_{u_n}).
\end{equation} 
\end{lemma}

\begin{proof}
\textit{Step 1 ($N=1$).} We suppose we are in dimension $N=1$ and, without loss of generality, that \eqref{eqn:energyBoundSlice} holds with the $\liminf$ replaced by a limit. Let $x_0\in J_u$. We \emph{claim} that for any $\eta>0$, for all sufficiently large $n$, $J_{u_n}\cap B(x_0,\eta)\neq \emptyset$. If on the contrary, $J_{u_n}\cap B(x_0,\eta)= \emptyset$ for a sequence $n\to \infty$, then we have that $u_n - \dashint_{B(x_0,\eta)} u_n$ is uniformly bounded in $W^{1,p}(B(x_0,\eta))$ by the Poincar\'e inequality and \eqref{eqn:energyBoundSlice}. Applying the Rellich--Kondrachov compactness theorem, it necessarily follows that $u \in W^{1,p}(B(x_0,\eta))$, contradicting that $x_0\in J_u$ and proving the claim. Take a finite subset $\Xi$ of $J_u$ and choose $\eta>0$ sufficiently small so that $B(x_0,\eta)$ and $B(x_1,\eta)$ are disjoint for any $x_0,x_1\in \Xi$. We apply the claim to find
\begin{equation}\nonumber
\mathcal{H}^0(J_u\cap \Xi) = \sum_{x\in \Xi} 1 \leq \liminf_{n\to \infty}\sum_{x\in \Xi} \mathcal{H}^0(J_{u_n}\cap B(x,\eta)) \leq \liminf_{n\to \infty}\mathcal{H}^0(J_{u_n}).
\end{equation}
Taking $\Xi \uparrow J_u$ concludes the one-dimensional case.

\textit{Step 2 ($N\geq 2$).} We provide the main ideas for completing the proof via slicing and refer to \cite{chambolleCrismaleEquilibrium,stinsonWittig} for related estimates via slicing. For a given direction $\xi\in \mathbb{S}^{N-1}$, we can define the orthogonal plane $\Pi_\xi : = \{y\in \R^N : \langle y,\xi \rangle  = 0\}$, and for each $y\in \Pi_\xi$, we can slice any set $\Omega$ into $\Omega_{\xi,y} : = \{t\in \R: y+t\xi \in \Omega\}$. We also define $u_{\xi,y}(t) : = u(y+t\xi)$. By \cite[Theorem 4.35]{AFP}, for $\xi\in \mathbb{S}^{N-1}$ and $\mathcal{H}^{N-1}$-a.e. $y\in \Pi_\xi$, $(u_n)_{\xi,y} \in GSBV^p(\Omega_{\xi,y}')$ and  $J_{(u_n)_{\xi,y}} = (J_{u_n})_{\xi,u},$ with the analogous relations holding for $u$. 

We apply Fubini's theorem and the coarea formula to find
\begin{equation}\nonumber
\int_{\Pi_\xi} \left[\int_{\Omega_{\xi,y}'} | (u_n)_{\xi,y}'|^p \, dt + \mathcal{H}^{0}(J_{(u_n)_{\xi,y}})\right]\, d\mathcal{H}^{N-1}(y) \leq \int_{\Omega'} |\nabla u_n|^p \, dx +\mathcal{H}^{N-1}(J_{u_n}).
\end{equation}
Applying Fatou's lemma, we see that 
$\liminf_{n\to \infty}\int_{\Omega_{\xi,y}'} | (u_n)_{\xi,y}'|^p \, dt + \mathcal{H}^{0}(J_{(u_n)_{\xi,y}}) < \infty$ and, from similar reasoning, that $(u_\eps)_{\xi,y}$ converges to $u_{\xi,y}$ in measure for $\mathcal{H}^{N-1}$-a.e. $y\in \Pi_\xi$. Thus, by Step 1 for $\mathcal{H}^{N-1}$-a.e. $y\in \Pi_\xi$, we have that 
\begin{equation}\nonumber
\mathcal{H}^{0}(J_{u_{\xi,y}}) \leq \liminf_{n\to \infty} \mathcal{H}^{0}(J_{(u_n)_{\xi,y}}).
\end{equation}
Integrating this relation over $y$, using $J_{(u_n)_{\xi,y}} = (J_{u_n})_{\xi,u}$, Fatou's lemma, and the coarea formula twice, we find
\begin{equation}\nonumber
\int_{J_u}|\langle \nu(x) ,\xi \rangle | \, d \mathcal{H}^{N-1}(x) = \int_{\Pi_\xi} \mathcal{H}^{0}((J_{u})_{\xi,y}) \, d\mathcal{H}^{N-1}(y) \leq \liminf_{n\to \infty} \mathcal{H}^{N-1}(J_{u_n})
\end{equation}
where $\nu(x)$ is the measure-theoretic normal for the surface $J_u$ at $x.$
We proved this inequality on all of $\Omega'$, but the argument is valid on any open subset $\mathcal{U}\subset \Omega'$ so that
 \begin{equation}\label{eqn:locLSCjump}
\int_{J_u\cap \mathcal{U}}|\langle \nu(x) ,\xi \rangle | \, d \mathcal{H}^{N-1}(x) \leq \liminf_{n\to \infty} \mathcal{H}^{N-1}(J_{u_n}\cap \mathcal{U}).
\end{equation}
A classical localization argument for measures (see \cite[Theorem 1.16 and Section 4.1]{braidesFreeDiscont}) allows one to locally optimize $\xi\in \mathbb{S}^{N-1}$ in \eqref{eqn:locLSCjump} to recover $\mathcal{H}^{N-1}(J_{u})\leq \liminf_{n\to \infty}\mathcal{H}^{N-1}(J_{u_n})$, as desired.
\end{proof}


\section*{Acknowledgements}

W.F. was supported by the NSF grant DMS-2407235.  K.S. thanks Manuel Friedrich for (the always) insightful comments and pointing to the reference \cite{DMFraToa_Preprint}. K.S. was also supported by funding from the NSF (USA) RTG grant DMS-2136198.


\bibliographystyle{amsplain}
\bibliography{./FS-compactness-arxiv}

\end{document}